\documentclass[12pt,leqno]{amsart}

\usepackage{amsmath,amsfonts,amssymb,amsthm,mathrsfs}
\usepackage{curves}  %%%% Dibujos
\usepackage{epic}  %%%% Dibujos
\usepackage[all]{xy}   %%%% Diagramas
\usepackage{graphicx}
\usepackage{epsfig}

\usepackage[usenames]{color}

\newtheorem{thm}{Theorem}[section]
\newtheorem{prop}[thm]{Proposition}

\newtheorem{cor}[thm]{Corollary}
\newtheorem{lem}[thm]{Lemma}

\newtheorem{rem}[thm]{Remark}
\newtheorem{defn}[thm]{Definition}

\DeclareMathOperator{\diam}{diam}

\newcommand{\Z}{{\mathbb Z}}      \newcommand{\R}{{\mathbb R}}

\def\dist{\qopname\relax o{dist}}
\def\b{\qopname\relax o{b}}

\title[Fractional inequalities in unbounded John domains]{Fractional Sobolev--Poincar\'e and fractional Hardy \\ inequalities in unbounded John domains}

\author{Ritva Hurri-Syrj\"anen and Antti V. V\"ah\"akangas}
\address{Department of Mathematics and Statistics, 
Gustaf H\"allstr\"omin katu 2$\b$, FI-00014 University of Helsinki, Finland}
\email{ritva.hurri-syrjanen@helsinki.fi}
\email{antti.vahakangas@helsinki.fi}
\date{\today}

\begin{document}


\begin{abstract}
We prove fractional Sobolev--Poincar\'e inequalities 
in unbounded John domains
and we characterize  fractional Hardy inequalities there.
\end{abstract}

\maketitle

\markboth{\textsc{R. Hurri-Syrj\"anen and A. V. V\"ah\"akangas}}
{\textsc{Fractional Sobolev--Poincar\'e  and fractional Hardy inequalities}}

\section{Introduction}

Let $D$ be a bounded $c$-John domain in $\R^n$, $n\ge 2$. Let numbers
$\delta,\tau \in (0,1)$ and exponents $p,q\in [1,\infty)$ be given such that $1/p-1/q = \delta/n$.
Then there is a  constant $C=C(\delta,\tau,p,n,c)$ such that the fractional Sobolev--Poincar\'e 
inequality 
\begin{equation}\label{fractionalqp}
\int_D\vert u(x)-u_D\vert ^q\,dx
\le
C
\biggl(\int_D\int_{B^n(x,\tau \dist(x,\partial D))}\frac{\vert u(x)-u(y)\vert ^p}{\vert x-y\vert ^{n+\delta p}}\,dy\,dx
\biggr)^{q/p}
\end{equation}
holds for all functions $u \in L^1(D)$.
For a proof we refer the reader to \cite[Theorem 4.10]{H-SV} when $1<p<n/\delta$ and to \cite{Dyda3} when $p=1$.

We prove the inequality corresponding   to
\eqref{fractionalqp} in unbounded
John domains,
Theorem 
\ref{t.main}.
The classical Sobolev--Poincar\'e inequality for an unbounded $c$-John domain $D$ has been proved in
\cite[Theorem 4.1]{MR1190332}: 
there is a finite constant $C(n,p,c)$ such that the inequality
\[
\inf_{a\in\R}Ê\int_D \lvert u(x)-a\rvert^{np/(n-p)}\,dx \le C(n,p,c)\bigg(\int_D \lvert \nabla u(x)\rvert^p\,dx\bigg)^{n/(n-p)}
\]
holds for all
$u\in L^1_p(D)=\{u\in\mathscr{D}'(D)\,:\,\nabla u\in L^p(D)\}$;
here $1\le p<n$.
We obtain the fractional Sobolev inequalities
\eqref{frac_Sobolev_inequality}
in unbounded John domains too, Theorem \ref{t.main_emb}.

As an application of the fractional Sobolev inequalities 
we characterize the
fractional Hardy inequalities 
\[
\int_{D} \frac{\lvert u(x)\rvert^q}{\dist(x,\partial D)^{q(\delta+n(1/q-1/p))}}\,dx
\le  C\bigg(\int_{D} \int_{D}
\frac{\lvert u(x)-u(y)\rvert ^p}{\lvert x-y\rvert ^{n+\delta p}}\,dy\,dx\bigg)^{q/p}
\]
in unbounded John domains $D$
whenever $\delta\in (0,1)$ and exponents $p,q\in [1,\infty)$ are given such that $p<n/\delta$ and $0\le 1/p-1/q\le \delta/n$ and the constant $C$ does not depend on  $u\in C_0(D)$, Theorem \ref{t.hardy_c}.
We also give sufficient geometric conditions for the fractional Hardy inequalities in Corollary \ref{t.cor}.
%It is well known that functions in $L^1_p(D$ are locally
%$L^p$-integrable.
%\subset L^p_{\textup{loc}}(D)$

\subsection*{Acknowledgements}

The authors thank Professor Vladimir Maz'ya for bringing up the question about fractional inequalities in unbounded domains.

\section{Notation and preliminaries}\label{s.notation}

Throughout the paper we assume that $D$ is a  domain and $G$ is an open set in the Euclidean $n$-space
$\R^n$, $n\geq 2$. 
The open ball centered at $x\in \R^n$ and with radius $r>0$ is  $B^n(x,r)$.
The Euclidean
distance from $x\in G$ to the boundary of $G$ is written as $\dist(x,\partial G)$.
The diameter of a set $A$ in $\R^n$ is $\mathrm{diam}(A)$.
The Lebesgue $n$-measure of a  measurable set $A$ is denoted by $\vert A\vert.$
For a measurable set $A$ with finite and positive measure and for an integrable function $u$ on $A$ the integral average is written as
\[
u_A=\frac{1}{\lvert A \rvert} \int_{A}u(x)\,dx\,.
\]
We write $\chi_A$ for the characteristic function of a set $A$.
For a proper open set $G$ in $\R^n$ we fix a Whitney decomposition $\mathcal{W}(G)$. 
The construction and the properties of Whitney cubes can be found in 
\cite[VI 1]{S}.
The family $C_0(G)$ consists of all continuous functions $u:G\to \R$ with a compact support in $G$. 
We let $C(\ast,\dotsb,\ast)$  denote a constant which depends on the quantities appearing
in the parentheses only.

We define the $c$-John domains so that unbounded domains are allowed, too. For other equivalent definitions we refer the reader to \cite{MR1246886} and \cite{MR1190332}.

\begin{defn}\label{sjohn}
A domain $D$ in $\R^n$ with $n\ge 2$ is a {\em $c$-John domain}, $c\ge 1$, if
each pair of points $x_1,x_2\in D$ can be joined
by a rectifiable curve $\gamma:[0,\ell]\to D$ parametrized by its arc length such that
  $\dist(\gamma(t),\partial D)\ge \min\{t,\ell -t\}/c$ for every $t\in [0,\ell]$.
\end{defn}

\noindent
Examples of unbounded John domains are the Euclidean 
$n$-space $\R^n$ and the infinite cone
\[
\big\{(x',x_n)\in\R^n \,:\,  x_n>\lVert x' \rVert\big\}\,.
\]
%The complement $\R^2\setminus K$ of the von Koch curve %$K$ is an unbounded John domain in the plane.
For more examples we refer the reader to 
\cite[4.3 Examples]{MR1190332}.

We recall a useful property of bounded John domains
from \cite[Theorem 3.6]{MR1246886}.

\begin{lem}\label{t.equi}
 Let $D$ in $\R^n$ be a bounded $c$-John domain, $n\ge 2$. 
 Then there exists a central point $x_0\in D$  such that every
  point $x$ in $D$ can be joined to $x_0$ by a rectifiable curve
  $\gamma:[0,\ell]\to D$, parametrized by its arc length, with
  $\gamma(0)=x$, $\gamma(\ell)=x_0$, and 
$\dist(\gamma(t),\partial D)\ge t/4c^2$ for each $t\in [0,\ell]$.
\end{lem}

The following engulfing property is in \cite[Theorem 4.6]{MR1246886}.

\begin{lem}\label{t.engulfing}
A $c$-John domain $D$ in $\R^n$ can be written as the union
of domains $D_1,D_2,\ldots$ such that
\begin{itemize}
\item[(1)]Ê$\overline{D_i}$ is compact in $D_{i+1}$ for each $i=1,2,\ldots$,
\item[(2)]Ê$D_i$ is a $c_1$-John domain for each $i=1,2,\ldots$ with $c_1=c_1(c,n)$.
\end{itemize}
\end{lem}

We define the upper and lower Assouad dimension of a
given set $E\not=\emptyset$ in $\R^n$.
The upper Assouad dimension measures how thin a given set is and the
lower Assouad dimension measures its fatness.
For further
discussion on these dimensions we refer to \cite[\S1]{KLV}.

\begin{defn}\label{def:uAssouad}
The upper Assouad dimension
of $E$, written as $\overline{\mathrm{dim}}_A(E)$,
is defined as the infimum of all numbers $\lambda\ge 0$ as follows:
There exists a constant $C=C(E,\lambda)> 0$ 
such that for every
$x\in E$ and for all $0<r<R<2\mathrm{diam}(E)$ the set 
$E\cap B^n(x,R)$  
can be covered by 
at most $C(R/r)^\lambda$ balls that are centered in $E$ and have radius $r$. 
\end{defn}

\begin{defn}
The lower Assouad dimension of $E$,
written as $\underline{\mathrm{dim}}_{A}(E)$, 
is defined as the supremum of all numbers  $\lambda\ge 0$ as follows:
There exists a constant $C=C(E,\lambda)> 0$ such that  for every $x\in E$ and  for all $0<r<R<2\diam(E)$ 
at least $C(R/r)^\lambda$ balls centered in $E$ and with radius $r$ are needed to cover the set $B^n(x,R)\cap E$. 
\end{defn}

Let $G$ be an open set in $\R^n$. Let $0< p<\infty$ and $0<\tau ,\delta<1$ be given. We write
\[
\lvert u \rvert_{W^{\delta,p}(G)} = \bigg( \int_G\int_{G}\frac{\lvert u(x)-u(y)\rvert^p}{\lvert x-y\rvert^{n+\delta p}}\,
dy\,dx\,\bigg)^{1/p}
\]
and
\[
\lvert u \rvert_{W^{\delta,p}_\tau(G)} = \bigg( \int_G\int_{B^n(x,\tau \mathrm{dist}(x,\partial G))}\frac{\lvert u(x)-u(y)\rvert^p}{\lvert x-y\rvert^{n+\delta p}}\,
dy\,dx\,\bigg)^{1/p}
\]
for appropriate measurable  functions $u$ on $G$.
When $G=\R^n$ both of the integrals in the latter form are taken over the whole space.
The homogeneous fractional Sobolev space
$\dot{W}^{\delta,p}_\tau(G)$ consists
of all measurable functions $u:G\to \R$  
with $\lvert u\rvert_{W^{\delta,p}_\tau(G)}<\infty$. 

The following lemma tells that the functions
$u\in \dot{W}^{\delta,p}_\tau(G)$
are locally $L^p$-integrable in $G$, that is 
$u\in L^p_{\textup{loc}}(G)$. We improve this
for John domains in Corollary \ref{embed}.

\begin{lem}\label{l.integrability}
Suppose that $G$ is an open set in $\R^n$. Let $0< p<\infty$ and $0<\tau,\delta<1$ be given.
Let $K$ be a compact set in $G$.  
If $u\in \dot{W}^{\delta,p}_\tau(G)$ then $u\in L^p(K)$. 
\end{lem}

\begin{proof}
We may assume that $G\not=\R^n$. If $G=\R^n$, then we just remove one point from $G\setminus K$.
By covering  $K$ with a finite number of
balls $B$ such that $\overline{B}\subset G$ we may assume that 
$K$ is the closure of such a ball.
Let us fix $\varepsilon>0$ such that $\varepsilon\tau/(1-\varepsilon \tau)<\tau$.
We obtain
\begin{equation}
\label{base_fin}
\begin{split}
\int_K\int_{K\cap B^n(z,\tau \mathrm{dist}(z,\partial G))}&  \lvert u(z)-u(y)\rvert^p\,dy\,dz \\&\le \mathrm{diam}(K)^{n+\delta p}\int_K\int_{K\cap B^n(z,\tau \mathrm{dist}(z,\partial G))} \frac{\lvert u(z)-u(y)\rvert^p}{\lvert z-y\rvert^{n+\delta p}}\,dy\,dz\\&\le \diam(K)^{n+\delta p}\lvert u\rvert_{W^{\delta,p}_\tau(G)}^p<\infty\,.
\end{split}
\end{equation}
Let us fix $x\in K$ and $0<r_x<\varepsilon \tau \dist(x,\partial G)$. Since
$K$ is the closure of  some ball, we have 
the inequality $\lvert K\cap B^n(x,r_x)\rvert>0$. 
By our estimates in \eqref{base_fin} there is a point $z_x\in K\cap B^n(x,r_x)$ so that
\begin{equation}\label{finite}
\int_{K\cap B^n(z_x,\tau \mathrm{dist}(z_x,\partial G))}  \lvert u(z_x)-u(y)\rvert^p\,dy < \infty\,.
\end{equation}
%That is $u\in L^p(K\cap B^n(z_x,\tau \mathrm{dist}(z_x,\partial G))$.
By the choice of $\varepsilon$ we have $x\in B^n(z_x,\tau \dist(z_x,\partial G))$ for
each $x\in K$. Thus,
\[
K\subset \bigcup_{x\in K} B^n(z_x,\tau \dist(z_x,\partial G))\,.
\]
By the compactness of the set $K$ there are points $x_1,\ldots,x_N$ in $K$ such  that
$K$ is contained in the union of the balls $B^n(z_i,\tau \dist(z_i,\partial G))$, where $z_i=z_{x_i}$ for
each $i$. Hence, by inequality \eqref{finite} we obtain
\begin{align*}
\int_K \lvert u(y)\rvert^p\,dy &\le \sum_{i=1}^N \int_{K\cap B^n(z_i,\tau \dist(z_i,\partial G))} \lvert u(y)\rvert^p\,dy
\\
&\le 2^p\sum_{i=1}^N \int_{K\cap B^n(z_i,\tau \dist(z_i,\partial G))} \lvert u(z_i)\rvert^p + \lvert u(z_i)-u(y)\rvert^p\,dy<\infty\,.
\end{align*}
This concludes the proof.
\end{proof}

The following definition is from \cite[\S1]{MR2317850}.
It arises from  generalized 
Poincar\'e inequalities that are studied in \cite[\S 7]{MR1609261}.
Let us fix $\kappa \ge 1$ and an open set $G$ in $\R^n$. For  $\delta\in [0,1]$,  $0<p\le \infty$, 
and $u\in L^1_{\textup{loc}}(G)$ we write
\[
\lvert u\rvert_{A^{\delta,p}_\kappa(G)} = \sup_{\mathcal{Q}_\kappa(G)} \Bigg\lVert \sum_{Q\in\mathcal{Q}_\kappa(G)} \bigg(\frac{1}{\lvert Q\rvert^{1+\delta/n}}\int_Q \lvert u(x)-u_Q\rvert\,dx \bigg)\chi_Q\Bigg\lVert_{L^p(G)}\,,
\]
where the supremum is taken over all families of cubes $\mathcal{Q}_\kappa(G)$ such that $\kappa Q\subset G$
for every $Q\in\mathcal{Q}_\kappa(G)$ and $Q\cap R = \emptyset$ if $Q$ and $R$ belong
to $\mathcal{Q}_\kappa(G)$ and $Q\not=R$.

\begin{lem}\label{embed_l}
Suppose that $G$ is an open set in $\R^n$.
Let $0<\tau,\delta<1$ and $1\le p<\infty$ be given.
Then there is a constant $\kappa=\kappa(n,\tau)\ge 1$ such that inequality
\begin{equation}\label{e.embed}
\lvert u\rvert_{A^{\delta,p}_\kappa(G)}\le (\sqrt n)^{n/p+\delta}\lvert u\rvert_{W^{\delta,p}_\tau(G)}
\end{equation}
holds for every $u\in L^1(G)$.
\end{lem}

\begin{proof} Let us choose $\kappa=\kappa(n,\tau)\ge 1$ such that $Q\subset B^n(x,\tau \dist(x,\partial G))$ whenever $x\in Q\in\mathcal{Q}_\kappa(G)$.
Then we fix a family of cubes $\mathcal{Q} :=\mathcal{Q}_\kappa(G)$. By Jensen's inequality we obtain
\begin{align*}
%\lvert u\rvert_{A^{\delta,p}(G)}^p &\le (1+\varepsilon)
\sum_{Q\in\mathcal{Q}} \lvert Q\rvert \bigg(\frac{1}{\lvert Q\rvert^{1+\delta/n}}\int_Q \lvert u(x)-u_Q\rvert\,dx\bigg)^p
\le \sum_{Q\in\mathcal{Q}} \lvert Q\rvert^{-\delta p/n}\int_Q \lvert u(x)-u_Q\rvert^p\,dx\,.
\end{align*}
By using Jensen's inequality again
\begin{align*}
&\sum_{Q\in\mathcal{Q}} \lvert Q\rvert^{-\delta p/n}\int_Q \lvert u(x)-u_Q\rvert^p\,dx\\&\le  (\sqrt n)^{n+\delta p}\sum_{Q\in\mathcal{Q}} \int_Q\int_Q  \frac{\lvert u(x)-u(y)\rvert^p}{\lvert x-y\rvert^{n+\delta p}}\,dy\,dx \le (\sqrt n)^{n+\delta p}\lvert u\rvert_{W^{\delta,p}_\tau(G)}^p\,.
\end{align*}
Taking supremum over all families $\mathcal{Q}_k(G)$ gives inequality \eqref{e.embed}.
\end{proof}

\section{Inequalities in bounded John domains}\label{s.bounded}

We give the following fractional Sobolev--Poincar\'e inequality in bounded John domains. 
The inequality for $p>1$ is already in \cite[Theorem 4.10]{H-SV}, but we need a better control over
the dependencies of the constant $C$. 

\begin{thm}\label{t.sub}
Suppose that $D$ is a bounded $c$-John domain in $\R^n$, $n\ge 2$. Let 
$\tau,\delta\in (0,1)$ and $1\le p< n/\delta$ be given. Then there is a constant
$C=C(\delta,\tau,p,n,c)>0$ such that
the fractional Sobolev--Poincar\'e inequality 
\begin{equation}\label{fractionalq1}
\int_D\vert u(x)-u_D\vert ^{np/(n-\delta p)}\,dx
\le
C
\lvert u \rvert_{W^{\delta,p}_\tau(D)}^{np/(n-\delta p)}
\end{equation}
holds for every $u\in L^1(D)$.
\end{thm}

Theorem \ref{t.sub} follows from Proposition \ref{WeakEquivToStrong} and Proposition \ref{p.weak}.
The following  result from \cite{Dyda3}, based upon the Maz'ya truncation
method \cite{Maz} adapted to the fractional setting,
shows that it is enough to prove a weak fractional Sobolev--Poincar\'e inequality.

\begin{prop}\label{WeakEquivToStrong}
Suppose that $G$  is an open set in $\R^n$ with $\lvert G\rvert<\infty$.
Let $0<\delta,\tau <1$ and 
$0<p\le q<\infty$ be given.
Then the following conditions are equivalent.
\begin{itemize}
\item[(A)]
There is a constant $C_1>0$ such that inequality
\begin{align*}
&\quad\inf_{a\in\R}\sup_{t>0}\lvert \{x\in G:\,|u(x)-a|>t\}\rvert t^{q} 
 \\&\qquad\qquad\qquad  \le C_1
\biggl(\int_{G}\int_{B^n(y,\tau \dist(y,\partial G))}\frac{\vert u(y)-u(z)\vert^p}{\vert y-z\vert ^{n+\delta p}}
\,dz\,dy\biggr)^{q/p}
\end{align*}
holds for  every $u\in L^\infty(G)$.
\item[(B)]
There is a constant $C_2>0$ such that inequality
\begin{align*}
\quad \inf_{a\in\R} \int_G\vert u(x)-a\vert ^{q}\,dx
\le
C_2
\bigg(\int_G\int_{B^n(y,\tau \dist(y,\partial G))}\frac{\vert u(y)-u(z)\vert^p}{\vert y-z\vert ^{n+\delta p}}\,dz
\,dy\bigg)^{q/p}
\end{align*}
holds for every $u\in L^1(G)$.
\end{itemize}
In the implication from (A) to (B)  $C_2=C(p,q)C_1$ and  from (B) to (A) 
 $C_1=C_2$. 
\end{prop}

The weak fractional Sobolev--Poincar\'e inequalities
hold in bounded John domains by the following proposition.

\begin{prop}\label{p.weak}
Suppose that $D$ is a bounded $c$-John domain in $\R^n$.
Let 
$\tau,\delta\in (0,1)$ and $1\le p< n/\delta$ be given. Then there is a constant
$C=C(\delta,\tau,p,n,c)>0$ such that
the weak fractional Sobolev--Poincar\'e inequality 
\begin{equation*}
\inf_{a\in\R} \sup_{t>0} \vert \{x\in D\,:\, \lvert u(x)-a\rvert  > t \}\rvert t^{np/(n-\delta p)}\le
C
\lvert u \rvert_{W^{\delta,p}_\tau(D)}^{np/(n-\delta p)}
\end{equation*}
holds for every $u\in L^\infty(D)$.
\end{prop}

For a simple proof of Proposition \ref{p.weak} we refer to
\cite[Theorem 4.10]{H-SV}. The dependencies of
the constants appearing in \cite[Theorem 4.10]{H-SV} can be tracked more
explicitly in order to obtain Proposition \ref{p.weak}.
In the present paper, we give 
a more general argument that might be of independent interest.

\medskip
The following Theorem \ref{iso} is the key result for proving Proposition \ref{p.weak}.

\begin{thm}\label{suff_thm}\label{iso}
Suppose that $D$ is a bounded $c$-John domain in $\R^n$. Let $\kappa\ge 1$ be fixed.
Let $\delta\in [0,1]$ and $1\le  p<n/\delta$ be given.
Then there exists a constant $C=C(n,\kappa,p,\delta,c)$ such that the inequality
\begin{equation}\label{fract}
\inf_{a\in\R} \sup_{t>0}Ê\lvert \{x\in D\,:\, \lvert u(x)-a\rvert > t\}\rvert t^{np/(n-\delta p)} \le C\lvert u\rvert_{A^{\delta,p}_\kappa(D)}^{np/(n-\delta p)}
\end{equation}
holds for every $u\in L^1(D)$.
\end{thm}

We give the proof of Theorem \ref{iso}
in Section \ref{s.iso_proof}.
By using Theorem \ref{iso} the claim of  Proposition  \ref{p.weak}
follows easily.

\begin{proof}[Proof of Proposition  \ref{p.weak}]
 
By Lemma~\ref{embed_l} it is enough to prove that there is a constant $C=C(\delta , \tau , p, n, c)$ such that the inequality
\[
\inf_{a\in\R} \sup_{t>0} \vert \{x\in D\,:\, \lvert u(x)-a\rvert >t \}\rvert t^{np/(n-\delta p)}\le 
C\lvert u\rvert_{A^{\delta,p}_{\kappa(n,\tau)}(D)}^{np/(n-\delta p)}
\]
holds for all $u\in L^\infty(D)$.
This inequality follows from  Theorem \ref{iso}.
\end{proof}

\section{Proof of Theorem \ref{iso}}\label{s.iso_proof}
We start to build up the proof for Theorem \ref{iso}
by giving auxiliary results.
The following lemma gives local inequalities.
Similar results are known in metric measure spaces, \cite[Theorem 4.1]{MR2317850}.

\begin{lem}\label{cube_lem}
Let $1\le  p,q<\infty$ be given
such that $1/p-1/q=\delta/n$ with $\delta\in [0,1]$. 
Then there is a  constant $C=C(n,p,\delta)>0$ such that inequality
\begin{equation}\label{assertion}
\sup_{t>0}\lvert \{x\in Q\,:\, \lvert u(x)-u_Q\rvert >t\}\rvert t^q\le C \lvert u\rvert_{A^{\delta,p}_1(Q)}^q
\end{equation}
holds for every cube $Q$ in $\R^n$ and for all $u\in L^1_{\textup{loc}}(\R^n)$.
\end{lem}

\begin{proof}
Let us fix $u\in L^1_{\textup{loc}}(\R^n)$. We write
for cubes $Q$ in $\R^n$
\begin{align*}
a(Q) &= \lvert u\rvert_{A^{\delta,p}_1(Q)}\cdot \lvert Q\rvert^{-1/q} \\&= \bigg\{ \lvert Q\rvert^{-p/q} \cdot \sup_{\mathcal{Q}_1(Q)} \sum_{R\in \mathcal{Q}_1(Q)} 
\lvert R\rvert^{1-\delta p/n} \bigg(\frac{1}{\lvert R\rvert}Ê\int_R \lvert u(x)-u_R\rvert\,dx\bigg)^p \bigg\}^{1/p}\,.
\end{align*}
Inequality \eqref{assertion} follows from
the generalized Poincar\'e inequality theorem \cite[Theorem 7.2(a)]{MR1609261} as soon as we prove inequalities \eqref{f_control} and \eqref{a_control}.
The inequality 
\begin{equation}\label{f_control}
\frac{1}{\lvert Q\rvert}\int_Q \lvert u(x)-u_Q\rvert\,dx \le a(Q)
\end{equation}
holds for every cube $Q$ in $\R^n$. Namely,
\begin{align*}
\frac{1}{\lvert Q\rvert}\int_Q \lvert u(x)-u_Q\rvert\,dx &= 
\bigg\{Ê\lvert Q\rvert^{-p/q} \cdot\lvert Q\rvert^{1-\delta p/n} \bigg(\frac{1}{\lvert Q\rvert}\int_Q \lvert u(x)-u_Q\rvert\,dx\bigg)^{p}\bigg\}^{1/p}\\&\le  a(Q)\,,
\end{align*}
because 
$1-p/q-\delta p/n=0$.
We need to show that
the inequality
\begin{equation}\label{a_control}
\sum_{P\in \mathcal{Q}_1(Q)} a(P)^q \lvert P\rvert \le 2^{q/p} a(Q)^q\lvert Q\rvert
\end{equation}
holds for all cubes $Q$ in $\R^n$ and all families $\mathcal{Q}_1(Q)$ of pairwise disjoint cubes inside $Q$.
In order to prove inequality \eqref{a_control} let us 
fix a cube $Q$ and its family $\mathcal{Q}_1(Q)$. 
For each $P\in\mathcal{Q}_1(Q)$ we fix  its family $\mathcal{Q}_1(P)$ such that
\[
\lvert u\rvert_{A^{\delta,p}_1(P)}^p\le 2\sum_{R\in \mathcal{Q}_1(P)} 
\lvert R\rvert^{1-\delta p/n} \bigg(\frac{1}{\lvert R\rvert}Ê\int_R \lvert u(x)-u_R\rvert\,dx\bigg)^p\,.
\]
Since $q/p\ge 1$,
\begin{align*}
\sum_{P\in\mathcal{Q}_1(Q)} a(P)^q \lvert P\rvert 
\le 2^{q/p}\bigg\{ \sum_{P\in\mathcal{Q}_1(Q)}  \sum_{R\in \mathcal{Q}_1(P)} 
\lvert R\rvert^{1-\delta p/n} \bigg(\frac{1}{\lvert R\rvert}Ê\int_R \lvert u(x)-u_R\rvert\,dx\bigg)^p \bigg\}^{q/p}\,.
\end{align*}
 Then writing
 $\mathcal{Q}:=\cup_{P\in\mathcal{Q}_1(Q)}Ê\mathcal{Q}_1(P)$ allows us to estimate
 \begin{align*}
 &\sum_{P\in\mathcal{Q}_1(Q)} a(P)^q \lvert P\rvert 
\\&\le 2^{q/p}\bigg\{ \sum_{R\in \mathcal{Q}} 
\lvert R\rvert^{1-\delta p/n} \bigg(\frac{1}{\lvert R\rvert}Ê\int_R \lvert u(x)-u_R\rvert\,dx\bigg)^p \bigg\}^{q/p}
\le 2^{q/p} a(Q)^q\lvert Q\rvert\,.
 \end{align*}
This implies inequality \eqref{a_control}.
\end{proof}

For a bounded $c$-John domain $D$ we let $\mathcal{W}^\kappa(D)$ be its modified Whitney decomposition
with a fixed $\kappa\geq 1$
such that
$\kappa Q^*=\kappa \tfrac 98 Q\subset D$ for each  $Q\in\mathcal{W}^\kappa(D)$.
This decomposition is obtained  by dividing each  Whitney cube $Q\in\mathcal{W}(D)$
into sufficiently small dyadic subcubes, their number depending on $\kappa$ and $n$ only.
The family of cubes in $\mathcal{W}^\kappa(D)$ with side length  $2^{-j}$, $j\in\mathbf{Z}$,
is written as
$\mathcal{W}_{j}^\kappa(D)$.

Let $Q$ be in $\mathcal{W}^\kappa(D)$.
Let us suppose that we are given a chain $\mathcal{C}(Q)\subset\mathcal{W}^\kappa(D)$ of cubes
\[\mathcal{C}(Q) = (Q_0,\ldots,Q_k)\,,\]
joining a fixed cube $Q_0\in\mathcal{W}^\kappa(D)$ to $Q_k=Q$ such that there exists a constant $C(n,\kappa )$
so that the inequality
\[
\lvert u_{Q^*} - u_{Q_0^*}\rvert  \le C(n,\kappa)\sum_{R\in\mathcal{C}(Q)} \frac{1}{\lvert R^*\rvert}\int_{R^*} \lvert u(x)-u_{R^*}\rvert\,dx
\]
holds whenever
$u\in L^1_{\textup{loc}}(D)$.
The family $\{\mathcal{C}(Q)\,:\,Q\in\mathcal{W}^\kappa(D)\}$ of chains of cubes is called a chain decomposition of $D$.
The shadow of a given cube $Q\in \mathcal{W}^\kappa(D)$ is the family
\[\mathcal{S}(R) = \{ Q\in\mathcal{W}^\kappa(D)\,:\, R\in\mathcal{C}(Q) \}\,.\]

The following key lemma is a straightforward modification of \cite[Proposition 2.5]{H-SMV} once we 
have Lemma \ref{t.equi}.

\begin{lem}\label{chain}
Let $D$ be a bounded $c$-John domain in $\R^n$.  Let $\kappa\ge 1$ and
$1\le q<\infty$ be given. Then 
there exist a chain
decomposition  of $D$ and constants $\sigma, \rho\in\mathbf{N}$ such that
\begin{itemize}
\item[(1)] $\ell(Q)\le 2^{\rho}\ell(R)$ for each $R\in\mathcal{C}(Q)$
  and $Q\in\mathcal{W}^\kappa(D)$,
\item[(2)] $\sharp \{R\in\mathcal{W}^\kappa_j(D):\ R\in\mathcal{C}(Q)\}\le
 2^ \rho$ for each $Q\in \mathcal{W}^\kappa(D)$ and $j\in \mathbf{Z}$,
\item[(3)] the inequality 
\begin{equation*}
\sup_{j\in \mathbf{Z}}\sup_{R\in\mathcal{W}^\kappa_j(D)}
\frac{1}{\lvert R\lvert}
\sum_{k=j-\rho}^\infty \sum_{\substack{Q\in \mathcal{W}^\kappa_k(D) \\ Q\in\mathcal{S}(R) }} \lvert Q\rvert
(\rho+1+k-j)^{q}  < \sigma
\end{equation*}
holds.
\end{itemize}
The constants $\sigma$ and $\rho$ depend 
on $\kappa$, $n$, $q$, and the John constant $c$ only.
\end{lem}

We are ready for the proof of Theorem \ref{suff_thm}.

\begin{proof}[Proof of Theorem \ref{suff_thm}]
Let us denote $q=np/(n-\delta p)$.
We need to show that 
there is a constant
$C(n,\kappa,p,\delta,c)$ such that
the inequality
\[
\inf_{a\in\R}\sup_{t>0} \lvert \{x\in D:\,|u(x)-a|>t\}\rvert t^q 
\le C(n,\kappa,p,\delta,c) \lvert u\rvert_{A^{\delta,p}_\kappa(D)}^q
\]
holds
for each $u\in L^1(D)$.
Let $Q_0$ be the fixed cube in the chain decomposition of $D$ given by Lemma \ref{chain}.
By the triangle inequality we obtain
\begin{align*}
  \lvert u(x)-u_{Q_0^*}\rvert &\le   \left|u(x)-\sum_{Q\in\mathcal{W}^\kappa(D)} u_{Q^*}\chi_{Q}(x)\right|+
  \left|\sum_{Q\in\mathcal{W}^\kappa(D)} u_{Q^*} \chi_Q(x) - u_{Q_0^*}\right|
\end{align*}
for almost every $x\in D$.
We write
\[
\left|u(x)-\sum_{Q\in\mathcal{W}^\kappa(D)} u_{Q^*}\chi_{Q}(x)\right|
 =: g_1(x) 
\]
and
\[
\left|\sum_{Q\in\mathcal{W}^\kappa(D)} u_{Q^*} \chi_Q(x) - u_{Q_0^*}\right|
 =: g_2(x) 
\]
for $x\in D$.
For a fixed $t>0$ we estimate
\begin{align*}
&  t^q \lvert \{x\in D:\ \lvert
  u(x)-u_{Q_0^*}\rvert>t\}\rvert \\&\qquad\le  
  t^q\left\lvert \{x\in D:\ g_1(x)>t/2\}\right\lvert + 
  t^q\left\lvert \{x\in D:\ g_2(x)>t/2\}\right\lvert \,. %t^q \mathbf{U}_1(t) +
%  t^q \mathbf{U}_2(t)\,.
\end{align*}
The local term $g_1$ is estimated by Lemma \ref{cube_lem} and the inequality $p\le q$:
\begin{align*}
  t^q\left\lvert \{x\in D:\ g_1(x)>t/2\}\right\lvert  & = \sum_{Q\in\mathcal{W}^\kappa(D)} t^q \lvert
  \{x\in \textup{int}(Q)\,:\, \lvert u(x)-u_{Q^*}\rvert >t/2\}\rvert 
  %\\&\le \sum_{Q\in\mathcal{W}^\kappa(D)}
%  t^q \lvert \{x\in Q^*:\ \lvert%
%  u(x)-u_{Q^*}\rvert>t/2\}\rvert 
\\& 
\le 
C 2^q\bigg(\sum_{Q\in\mathcal{W}^\kappa(D)} \lvert u\rvert_{A^{\delta,p}_1(Q^*)}^p\bigg)^{q/p}\,.
\end{align*}
Let us note that $\kappa R \subset \kappa Q^*\subset D$ if $R\in \mathcal{Q}_1(Q^*)$ and $Q\in\mathcal{W}^\kappa(D)$. 
We divide the family $\{Q^*\,:\,Q\in\mathcal{W}^\kappa(D)\}$ of cubes into
$C(n,\kappa)$ families so that each of them consists of pairwise disjoint cubes. As in the proof of 
Lemma \ref{cube_lem} we obtain
\[
 t^q \lvert \{x\in D: \ g_1(x)>t/2\}\rvert \le C\lvert u\rvert^q_{A^{\delta,p}_\kappa(D)}\,.
\]

We start to estimate the chaining term $g_2$:
\begin{align*}
   t^q\left\lvert \{x\in D:\ g_2(x)>t/2\}\right\lvert  & = t^q\sum_{Q\in\mathcal{W}^\kappa(D)} \lvert
  \{ x\in \textup{int}(Q)\,:\, \lvert u_{Q^*}-u_{Q_0^*}\rvert >t/ 2\}\rvert \\ &\le
%  &= \sum_{\substack {Q\in\mathcal{W}^\kappa(D) \\ \lvert
 %     u_{Q^*}-u_{Q_0^*}\rvert>t/2}}t^q\lvert Q\rvert\\&\le
 2^q \sum_{Q\in\mathcal{W}^\kappa(D)} \lvert Q\rvert\lvert
  u_{Q^*}-u_{Q_0^*}\rvert^q=:\Sigma\,.
\end{align*}
By property (1) of the chain decomposition in Lemma \ref{chain} we obtain
\begin{align*}
\Sigma &\le C\sum_{k=-\infty}^\infty \sum_{Q\in
    \mathcal{W}^\kappa_k(D)} \lvert Q\rvert\Bigg(\sum_{j=-\infty}^{k+\rho}
\underbrace{    \sum_{\substack{R\in\mathcal{W}^\kappa_j(D) \\ R\in \mathcal{C}(Q)}}
    \frac{1}{\lvert R^*\rvert}\int_{R^*} \lvert u(x)-u_{R^*}\rvert\,dx}_{=:\Sigma_{j,Q}}
  \Bigg)^q\\
&=C\sum_{k=-\infty}^\infty \sum_{Q\in
    \mathcal{W}^\kappa_k(D)} \lvert Q\rvert \Bigg(\sum_{j=-\infty}^{k+\rho} 
    \underbrace{(\rho+1+k-j)^{-1}(\rho+1+k-j)}_{=1}\Sigma_{j,Q}\Bigg)^q\,.
\end{align*}
Property (2)  in Lemma \ref{chain} and the equation $1/p-1/q=\delta/n$ give
 \begin{align*}
 \Sigma_{j,Q}^q & = \Bigg(\sum_{\substack{R\in\mathcal{W}^\kappa_j(D) \\ R\in \mathcal{C}(Q)}}
    \frac{1}{\lvert R^*\rvert}\int_{R^*} \lvert u(x)-u_{R^*}\rvert\,dx\Bigg)^q\\
&   \le C\sum_{\substack{R\in\mathcal{W}^\kappa_j(D) \\ R\in \mathcal{C}(Q)}}
\Bigg(    \frac{1}{\lvert R^*\rvert}\int_{R^*} \lvert u(x)-u_{R^*}\rvert\,dx\Bigg)^q
\le C\sum_{\substack{R\in\mathcal{W}^\kappa_j(D) \\ R\in \mathcal{C}(Q)}} 
\frac{\lvert u\rvert_{A^{\delta,p}_1(R^*)}^{q}}{\lvert R^*\rvert}\,.
 \end{align*}
%\begin{block}{}
Thus, H\"older's inequality and property (3) in Lemma \ref{chain} imply that
\begin{align*}
\Sigma& \leq C\sum_{k=-\infty}^\infty \sum_{Q\in
    \mathcal{W}^\kappa_k(D)} \lvert Q\rvert\sum_{j=-\infty}^{k+\rho} (\rho+1+k-j)^q
  \sum_{\substack{R\in\mathcal{W}^\kappa_j(D) \\ R\in
      \mathcal{C}(Q)}}\frac{\lvert u\rvert_{A_1^{\delta,p}(R^*)}^{q}}{\lvert
    R^*\rvert} \\
%  & = C\sum_{j=0}^\infty \sum_{R\in \mathcal{W}^\kappa_j(D)}
%  \mathcal{K}_u^q(R^*)\mathbf{T}_{j,R}, \\
  & = C\sum_{j=-\infty}^\infty \sum_{R\in \mathcal{W}^\kappa_j(D)}
  \lvert u\rvert_{A_1^{\delta,p}(R^*)}^{q} \cdot \frac{1}{\lvert R\rvert} \sum_{k=j-\rho}^\infty
  \sum_{\substack{Q\in\mathcal{W}^\kappa_k(D) \\ Q\in \mathcal{S}(R)}} 
  \lvert Q\rvert(\rho+1+k-j)^q  \\
  & %\leq C\sum_{j=-\infty}^\infty \sum_{R\in \mathcal{W}^\kappa_j(D)}
% \lvert u\rvert_{A^{\delta,p}(R^*)}^{q}
    \leq C\bigg(\sum_{j=-\infty}^\infty \sum_{R\in \mathcal{W}^\kappa_j(D)}
 \lvert u\rvert_{A^{\delta,p}_1(R^*)}^{p}\bigg)^{q/p} \leq C\lvert u\rvert_{A^{\delta,p}_\kappa(D)}^{q}\,.
\end{align*}
The theorem is proved.
\end{proof}

\section{Sobolev--Poincar\'e inequalities in unbounded John domains}\label{s.fractional}

We prove a fractional Sobolev--Poincar\'e inequality in unbounded John domains.

\begin{thm}\label{t.main}
Suppose that $D$ in $\R^n$ is an unbounded  $c$-John domain and that
$\tau,\delta\in (0,1)$ are given. Let $1\le p< n/\delta$. Then
there is a constant $C=C(\delta,\tau,p,n,c)>0$ such that
the fractional Sobolev--Poincar\'e inequality 
\begin{equation*}
\inf_{a\in\R}\int_D\vert u(x)-a\vert ^{np/(n-\delta p)}\,dx
\le
C
\lvert u \rvert_{W^{\delta,p}_\tau(D)}^{np/(n-\delta p)}
\end{equation*}
holds for each $u\in \dot{W}^{\delta,p}_\tau(D)$.
\end{thm}

The proof is similar to the proof of \cite[Theorem 4.1]{MR1190332}
where the classical Sobolev--Poincar\'e inequality has been proved in unbounded domains which have an engulfing
property. The proof is based on an idea from \cite{MR802488}.

\begin{proof}[Proof of Theorem \ref{t.main}]
By Lemma \ref{t.engulfing} the  $c$-John domain $D$ has an engulfing property.
That is, there are bounded $c_1$-John domains $D_i$  with $c_1=c_1(c,n)$ such that
\begin{equation*}
D_i\subset \overline{D_i}\subset D_{i+1}\,,\qquad  i=1,2,\dots\,,
\end{equation*}
and
\[
D=\bigcup_{i=1}^{\infty}D_i\,.
\]
Let us fix $u\in \dot{W}^{\delta,p}_\tau (D)$. By Lemma \ref{l.integrability}
with $K=\overline{D_i}$ we obtain that $u\in L^p(D_i)$ and hence $u\in  L^1(D_i)$ for each $i$. Therefore we may write
\begin{equation*}
u_i=u_{D_i}=\frac{1}{\vert D_i\vert}\int_{D_i}u(x)\,dx\,,\qquad  i=1,2,\dots.
\end{equation*}
The sequence $(u_i)$ is bounded.
Namely,
by the triangle inequality
\begin{equation*}
\vert u_i\vert = \frac{1}{\vert D_1\vert}\int_{D_1}\vert u_i\vert\,dx
\le \frac{1}{\vert D_1\vert}\biggl(
\int_{D_1}\vert u(x)-u_i\vert\,dx +\int_{D_1}\vert u(x)\vert \,dx\biggl)\,.
\end{equation*}
By H\"older's inequality with exponents $(np/(np-n+\delta p), np/(n-\delta p))$ and
by Theorem \ref{t.sub}  applied in a bounded $c_1$-John domain
$D_i$ we obtain
\begin{align*}
\int_{D_1}\vert u(x)-u_i\vert \,dx
&\le
\vert D_1\vert^{1-1/p+\delta /n}
\vert\vert u-u_{D_i}\vert\vert_{L^{np/(n-\delta p)}(D_1)}\\
&
\le
\vert D_1\vert^{1-1/p+\delta /n}
\vert\vert u-u_{D_i}\vert\vert_{L^{np/(n-\delta p)}(D_i)}
\le\vert D_1\vert^{1-1/p+\delta /n}
C
\lvert u \rvert_{W^{\delta,p}_\tau(D)}<\infty
\end{align*}
with a constant $C=C(\delta,\tau,p,n,c_1)$.

The bounded sequence $(u_i)$ has a convergent subsequence $(u_{i_j})$ and hence there is
a constant 
$a\in\R$ such that $\lim_{j\to\infty} u_{i_j}=a$.
By Fatou's lemma
and  Theorem \ref{t.sub} applied with the function $u\in L^p(D_j)$ we obtain
\begin{align*}
\int_D\vert u(x)- a\vert ^{np/(n-\delta p)}\,dx
&=\int_D\liminf_{j\to\infty}\chi_{D_{i_j}}(x)\vert u(x)-u_{i_j}\vert^{np/(n-\delta p)}\,dx\\
&\le\liminf_{j\to\infty}\int_{D_{i_j}}\vert u(x)-u_{i_j}\vert^{np/(n-\delta p)}\,dx\\
&\le C\liminf_{j\to\infty}\lvert u \rvert_{W^{\delta,p}_\tau(D_{i_j})}^{np/(n-\delta p)}
\le C\lvert u \rvert_{W^{\delta,p}_\tau(D)}^{np/(n-\delta p)}\,,
\end{align*}
where $C=C(\delta,\tau,p,n,c_1)$. The claim follows.
\end{proof}

A fractional Sobolev inequality holds in unbounded John domains.
%for compactly supported functions in $u\in \dot{W}^{\delta,p}_\tau(D)$.

\begin{thm}\label{t.main_emb}
Suppose that $D$ in $\R^n$ is an  unbounded $c$-John domain and that
$\tau,\delta\in (0,1)$ are given. Let $1\le p< n/\delta$. Then
there is a constant $C=C(\delta,\tau,p,n,c)>0$ such that
the fractional Sobolev inequality
\begin{equation}\label{frac_Sobolev_inequality}
\int_D\vert u(x)\vert ^{np/(n-\delta p)}\,dx
\le
C
\lvert u \rvert_{W^{\delta,p}_\tau(D)}^{np/(n-\delta p)}
\end{equation}
holds for each $u\in \dot{W}^{\delta,p}_\tau(D)$ with 
a compact support in $D$.
\end{thm}

\begin{proof}
We write $D=\cup_{i=1}^\infty  D_i$ as in the proof of Theorem \ref{t.main}.
By Lemma \ref{l.integrability} and the fact that the numbers $\lvert D_i\rvert$ converge to $ \lvert D\rvert=\infty$ as $i\to \infty$ we obtain that
\[
\lvert u_i\rvert= \lvert u_{D_i}\rvert \le \bigg( \frac{1}{\lvert D_i\rvert}\int_{D_i}Ê\lvert u(x)\rvert^p\,dx\bigg)^{1/p}\le \lvert D_i\rvert^{-1/p}Ê\lVert u\rVert_{L^p(D)}\xrightarrow{i\to \infty}Ê0
\]
for $u\in \dot{W}^{\delta,p}_\tau(D)$ with a compact
support in $D$.
Therefore, the  proof follows as the proof of Theorem \ref{t.main} with 
 $a=0$.
\end{proof}

The following is a corollary of Theorem \ref{t.main} and Theorem \ref{t.main_emb}.  
It shows that
$\dot W^{\delta,p}_\tau(D)$ is embedded to $L^{np/(n-\delta p)}(D)$ if we identify any two functions in $\dot W^{\delta,p}_\tau(D)$ whose difference
is a constant almost everywhere. This identification is usually included already in  the definition of homogeneous  spaces of smoothness  $0<\delta <1$.

%The following corollary is valid for bounded and unbounded John domains.

\begin{cor}\label{embed}
Suppose that $D$ in $\R^n$ is a  $c$-John domain and that
$\tau,\delta\in (0,1)$ are given. Let $1\le p< n/\delta$ and $q=np/(n-\delta p)$. Then
there is a nonlinear bounded operator 
\[
\mathbf{E}:\dot W^{\delta,p}_{\tau}(D)\to  L^q(D)\,,\qquad \mathbf{E}(u)=  u-a_u\,,
\]
whose norm is bounded by a constant $C=C(\delta,\tau,p,n,c)$; here $a_u\in \R$ for each $u\in \dot W^{\delta,p}_{\tau}(D)$.
If $D$ is an unbounded $c$-John domain, then $\mathbf{E}(u)=u$
for each $u\in \dot W^{\delta,p}_{\tau}(D)$ whose support is a compact set in $D$.
\end{cor}

\section{Fractional Hardy inequalities in unbounded John domains}

We characterize certain fractional  Hardy inequalities in unbounded
John domains 
as an application of   Theorem \ref{t.main_emb}. 
The following definition is motivated by the fractional Hardy inequalities from \cite{EH-SV}.
The classical $(q,p)$-Hardy inequalities are studied in
\cite{EH-S}.

We say that a fractional $(\delta,q,p)$-Hardy
inequality with $0<\delta<1$ and $0< p,q<\infty$ holds in a proper open set $G$ in $\R^n$, if there is a
constant $C>0$ such that
the inequality
\begin{equation}\label{e.hardy}
\int_{G} \frac{\lvert u(x)\rvert^q}{\dist(x,\partial G)^{q(\delta+n(1/q-1/p))}}\,dx
\le  C\bigg(\int_{G} \int_{G}
\frac{\lvert u(x)-u(y)\rvert ^p}{\lvert x-y\rvert ^{n+\delta p}}\,dy\,dx\bigg)^{q/p}
\end{equation}
holds for all functions $u\in C_0(G)$. The fractional Sobolev inequality \eqref{frac_Sobolev_inequality} is obtained when $1/p-1/q=\delta/n$.
The usual fractional $(\delta,p,p)$-Hardy inequality is obtained when $q=p$.

Our characterization of fractional Hardy inequalities is given in terms of Whitney cubes from $\mathcal{W}(G)$ 
and
the  $(\delta,p)$-capacities
\[
\mathrm{cap}_{\delta,p}(K,G) = \inf_u \lvert u\rvert_{W^{\delta,p}(G)}^p
\]
of compact sets $K$ in $G$,
where the infimum is taken over all  $u\in C_0(G)$ such that $u(x)\ge 1$ for each point $x\in K$.

\begin{thm}\label{t.hardy_c}
Let $D$ be an unbounded $c$-John domain in $\R^n$,  $D\not=\R^n$. Let
$\delta\in (0,1)$ and $1\le p,q<\infty$ be given such that  $p<n/\delta$ and $0\le 1/p-1/q\le \delta/n$.  Then the following conditions are
equivalent.
\begin{itemize}
\item[(A)]  A fractional  $(\delta,q,p)$-Hardy inequality holds in $D$.
\item[(B)]
There exists a positive constant $N>0$ such that  inequality
\begin{equation*}
\bigg(\sum_{Q\in\mathcal{W}(D)} \mathrm{cap}_{\delta,p}(K\cap Q,D)^{q/p}\bigg)^{p/q} \le N\,\mathrm{cap}_{\delta,p}(K,D)
\end{equation*}
holds for every compact set $K$ in $D$.
\end{itemize}
\end{thm}

The proof  of Theorem \ref{t.hardy_c} is based on the fractional Sobolev inequalities and  the Maz'ya type characterization 
for the validity
of a fractional $(\delta,q,p)$-Hardy inequality, 
Theorem \ref{t.maz'ya}.
Before the proof of Theorem \ref{t.hardy_c} 
we give some remarks, corollaries and auxiliary results.

\begin{rem}
There exist unbounded John domains where a fractional $(\delta,p,p)$-Hardy inequality fails
for some values of $\delta$ and $p$. As an example  let us define
$D=\R^2\setminus L$, where $L$ is a closed line-segment in $\R^2$. Then,  the fractional
$(\delta,p,p)$-Hardy inequality fails whenever $1< p<\infty$ and $\delta = 1/p$.
This example is a modification of  \cite[Theorem 9]{Dyda2}.
\end{rem}

Sufficient geometric conditions for the fractional 
Hardy inequalities are given in the following corollary. For the relevant notation we refer to Section \ref{s.notation}.

\begin{cor}\label{t.cor}
Let $D$ be an unbounded $c$-John domain in $\R^n$, $D\not=\R^n$.
Let $0<\delta<1$ and $1\le p,q<\infty$ be given such that $p<n/\delta$ and  
$0\le 1/p-1/q\le \delta/n$.
Then the fractional $(\delta,q,p)$-Hardy inequality \eqref{e.hardy}  holds in $D$ if either condition (A) or condition (B) holds.
\begin{itemize}
\item[(A)] $\overline{\mathrm{dim}}_A(\partial D) < n-\delta p$;
\item[(B)]  $\underline{\mathrm{dim}}_A(\partial D)> n-\delta p$ and $\partial D$ is unbounded.
\end{itemize}
\end{cor}

\begin{proof}
By 
Theorem \ref{t.hardy_c}
it is enough to prove a $(\delta,p,p)$-Hardy inequality which is a consequence of  \cite[Theorem 2]{Dyda1}.
The plumpness condition required there  
follows from the John condition in Definition \ref{sjohn}.
\end{proof}

Now we start to build up our proof for Theorem
\ref{t.hardy_c}.
First we give a characterization which is an extension
of \cite[Proposition 5]{Dyda2} where the special
case of $p=q$ is considered. 
This type of characterizations go back to Vl. Maz'ya,
\cite{Maz}.

\begin{thm}\label{p.maz'ya}
Suppose that $G$ is an open set 
in $\R^n$ and $\omega:G\to [0,\infty)$ is  measurable. Then the
following conditions are equivalent 
whenever $0<\delta<1$ and $0<p\le q<\infty$.
\begin{itemize}
\item[(A)] There is a constant $C_1>0$ such that 
the inequality
\[\int_G \lvert u(x)\rvert^q\,\omega(x)\,dx \le C_1 \lvert u\rvert_{W^{\delta,p}(G)}^q\]
holds for every $u\in C_0(G)$.
\item[(B)] There is a constant $C_2>0$ such that
the inequality
\[
\int_K \omega(x)\,dx \le C_2\, \mathrm{cap}_{\delta,p}(K,G)^{q/p}\
\]
holds
for every compact set $K$ in $G$.
\end{itemize}
In the implication from (A) to (B)  $C_2=C_1$ and from (B) to (A) 
$C_1 = \frac{C_2 2^{3q+2q/p}}{ (1-2^{-p})^{q/p}}$. 
\end{thm}

As a corollary of Theorem \ref{p.maz'ya}
we obtain
Theorem \ref{t.maz'ya} when we choose  
%\ref{p.maz'ya} 
\[
\omega=\mathrm{dist}(\cdot,\partial G)^{-q(\delta+n(1/q-1/p))}\,.
\]

\begin{thm}\label{t.maz'ya}
Let $0<\delta<1$ and $0<p\le q<\infty$. Then a $(\delta,q,p)$-Hardy inequality 
\eqref{e.hardy}  holds in a proper open set $G$ in $\R^n$ if and only if
there is a constant $C>0$ such that the inequality
\begin{equation}\label{e.maz'ya}
\int_K \mathrm{dist}(x,\partial G)^{-q(\delta+n(1/q-1/p))}\,dx \le C\,\mathrm{cap}_{\delta,p}(K,G)^{q/p}
\end{equation}
holds for every compact set $K$ in $G$.
\end{thm}

The proof for Theorem \ref{p.maz'ya} is a 
simple modification of the proof of \cite[Proposition 5]{Dyda2},
but we give the proof in the general case 
for the convenience of the reader.

\begin{proof}[Proof of Theorem \ref{p.maz'ya}]
Let us first assume that condition (A) holds.
Let $u\in C_0(G)$ be such that $u(x)\geq 1$ for every point $x\in K$.
By  condition (A) we obtain
\[
 \int_K \omega(x)\,dx \leq \int_G \lvert u(x)\rvert^q\,\omega(x)\,dx
\leq C_1 \bigg(\int_G
 \int_G \frac{\lvert u(x)-u(y)\rvert^p}{\lvert x-y\rvert^{n+\delta p}}\, dy\, dx\bigg)^{q/p}\,.
\]
Taking infimum over all such functions $u$ we obtain condition (B) with $C_2=C_1$.

Now let us assume that condition (B) holds and let $u\in C_0(G)$. We write
\[
 E_k = \{x\in G \,:\, \lvert u(x)\rvert > 2^k \} \quad \text{and} \quad A_k = E_k \setminus E_{k+1}\,, k\in \Z\,.
\]
Let us note that
\begin{equation}\label{e.decomp}
G= \{x\in G\,:\, 0\le \lvert u(x)\rvert <\infty\} =\{x\in G \,:\, u(x)=0\}\cup \bigcup_{i\in \Z} A_i\,.
\end{equation}
By condition (B) 
\begin{equation*}%\label{e.distcap}
\begin{split}
 \int_G \lvert u(x)\rvert^q \omega(x)\,dx
&\leq
 \sum_{k\in \Z} 2^{(k+2)q} \int_{A_{k+1}}  \omega(x)\,dx 
\\&\leq
C_2 2^{2q} \sum_{k\in \Z} 2^{kq} \mathrm{cap}_{\delta ,p}(\overline{A}_{k+1}, G)^{q/p}  \,.
\end{split}
\end{equation*}
Let us define $u_k:G\to [0,1]$ by
\[
u_k(x) = \begin{cases}
1, & \text{if $\lvert u(x)\rvert \geq 2^{k+1}$\,,}\\
\frac{\lvert u(x)\rvert}{2^k}-1 &\text{if $2^k < \lvert u(x)\rvert < 2^{k+1}$\,,}\\
0, & \text{if $\lvert u(x)\rvert \leq 2^k$\,.}
\end{cases}
\]
Then $u_k \in C_0(G)$ and $u_k(x)=1$ if $x\in \overline{E}_{k+1}$. We note that $\overline{A}_{k+1}
\subset\overline{E}_{k+1}$.
Thus we may take $u_k$ as a test function for the capacity. 
Let us write $F=\{x\in G \,:\, u(x)=0\}$. By \eqref{e.decomp},
\begin{align*}
&\mathrm{cap}_{\delta ,p}(\overline{A}_{k+1}, G) \leq 
\int_G \int_G \frac{\lvert u_k(x)-u_k(y)\rvert^p}{\lvert x-y\rvert^{n+\delta p}}\,dy\,dx \nonumber\\
&\leq
 2 \sum_{i\leq k} \sum_{j\geq k} \int_{A_i} \int_{A_j} \frac{\lvert u_k(x)-u_k(y)\rvert ^p}{\lvert x-y\rvert^{n+\delta p}}\,dy\,dx
+ 2 \sum_{j\geq k} \int_{F} \int_{A_j} \frac{\lvert u_k(x)-u_k(y)\rvert ^p}{\lvert x-y\rvert^{n+\delta p}}\,dy\,dx\,.
\end{align*}
The inequality
\begin{equation*}%\label{e.lip}
\lvert u_k(x)-u_k(y)\rvert \leq 2\cdot 2^{-j} \lvert u(x)-u(y)\rvert
\end{equation*}
holds whenever $(x,y)\in  A_i\times A_j$ and $i\leq k \leq j$.
Namely: $\lvert u_k(x)-u_k(y)\rvert \leq 2^{-k} \lvert u(x)-u(y)\rvert$ if $x,y\in G$. If $x\in A_i$ and $y\in A_j$
with $i+2\leq j$, then
$ \lvert u(x)-u(y)\rvert \geq \lvert u(y)\rvert -\lvert u(x)\rvert \geq 2^{j-1}$.
Hence $\lvert u_k(x) - u_k(y)\rvert \leq 1 \leq 2\cdot 2^{-j} \lvert u(x)-u(y)\rvert$.
\noindent
Thus, since $q\ge p$,
\begin{align*}
\sum_{k\in\Z}  2^{kq}\bigg( \sum_{i\leq k} &\sum_{j\geq k} \int_{A_i} \int_{A_j} \frac{\lvert u_k(x)-u_k(y)\rvert ^p}{\lvert x-y\rvert^{n+\delta p}}\,dy\,dx\bigg)^{q/p} \\
&\le 2^q\bigg( \sum_{k\in\Z}  \sum_{i\leq k} \sum_{j\geq k} 2^{(k-j)p} \int_{A_i} \int_{A_j} \frac{\lvert u(x)-u(y)\rvert ^p}{\lvert x-y\rvert^{n+\delta p}}\,dy\,dx\bigg)\,.
% \mathrm{cap}_{\delta ,p}(\overline{A}_{k+1}, G) \leq
% 2^{p+1} \sum_{i\leq k} \sum_{j\geq k} 2^{-jp} \int_{A_i} \int_{A_j} \frac{\lvert u(x)-u(y)\rvert^p}{\lvert x-y\rvert^{n+\delta p}}\,dy\,dx\,.
\end{align*}
By proceeding in a similar way as before we obtain that
\begin{align*}
\sum_{k\in\Z}  2^{kq}\bigg( &\sum_{j\geq k} \int_{F} \int_{A_j} \frac{\lvert u_k(x)-u_k(y)\rvert ^p}{\lvert x-y\rvert^{n+\delta p}}\,dy\,dx\bigg)^{q/p} \\
&\le 2^q\bigg( \sum_{k\in\Z}  \sum_{j\geq k} 2^{(k-j)p} \int_{F} \int_{A_j} \frac{\lvert u(x)-u(y)\rvert ^p}{\lvert x-y\rvert^{n+\delta p}}\,dy\,dx\bigg)^{q/p}\,.
\end{align*}

Using the sum of the geometric series $\sum_{k=i}^j 2^{(k-j)p} < \sum_{k=-\infty}^j 2^{(k-j)p} \le  \frac{1}{1-2^{-p}}$ and changing the order of summations gives
\[
\int_G \lvert u(x)\rvert^q\omega(x)\,dx
 \leq
\frac{C_2 2^{3q+2q/p}}{(1-2^{-p})^{q/p}}
\bigg(\int_G \int_G \frac{\lvert u(x)-u(y)\rvert^p}{\lvert x-y\rvert^{n+\delta p}}\, dy\, dx\bigg)^{q/p}\,.
\]
Thus condition (A) is true with $C_1 = C_2 2^{3q+2q/p} (1-2^{-p})^{-q/p}$. %frac{C_2 2^{3q+q/p}}{(1-2^{-p})^{q/p}}$.
\end{proof}

The following lemma is an extension of \cite[Proposition 6]{Dyda2}.

\begin{lem}\label{p.necessary}
Let $0<\delta <1$ and 
$0< p\le q<\infty$ be given.
Suppose that the  fractional $(\delta ,q,p)$-Hardy inequality \eqref{e.hardy} holds in a proper open set $G$ in $\R^n$
with a constant $C_1>0$. Then  there is a constant $N=N(C_1,n,\delta,q,p)>0$ such that the inequality
\begin{equation}\label{e.wanted}
\sum_{Q\in\mathcal{W}(G)} \mathrm{cap}_{\delta ,p}(K\cap Q,G)^{q/p}\le N^{q/p}\, \mathrm{cap}_{\delta,p}(K,G)^{q/p}
\end{equation}
holds for every compact set $K$ in $G$.
\end{lem}

\begin{proof}
If $Q\in\mathcal{W}(G)$ we write $\widehat{Q}=\tfrac {17}{16}Q$ and  $Q^*= \tfrac{9}{8}Q$. We recall that the side lengths of these cubes are comparable to their distances
from  $\partial G$.
 
Let us fix a compact set $K$ in $G$ and $u\in C_0(G)$ 
such that $u(x)\ge 1$ for each $x\in K$.
For every $Q\in\mathcal{W}(G)$ we 
let $\varphi_Q$ be a smooth function such that $\lvert \nabla \varphi_Q\rvert \le C\ell(Q)^{-1}$ and
$\chi_{Q}\le \varphi_Q\le \chi_{\widehat{Q}}$. Then, $u_Q:= u\varphi_Q$ is an admissible test function for $\mathrm{cap}_{\delta ,p}(K\cap Q,G)$. 
Hence, we can estimate the left hand side of inequality \eqref{e.wanted} by
\begin{align*}
&\sum_{Q\in\mathcal{W}(G)}
\bigg(\int_G \int_G \frac{\lvert u_Q(x)-u_Q(y)\rvert^{p}}{\lvert x-y\rvert^{n+\delta p}}\,dy\,dx\bigg)^{q/p}\\
&\le C \sum_{Q\in\mathcal{W}(G)}\bigg(    
\int_{\widehat{Q}} \frac{\lvert u_Q(x)\rvert^p}{\mathrm{dist}(x,\partial G)^{\delta p}}\,dx+
\int_{Q^*} \int_{Q^*}\frac{\lvert u_Q(x)-u_Q(y)\rvert^{p}}{\lvert x-y\rvert^{n+\delta p}}\,dy\,dx
 \bigg)^{q/p}\\
 &\le C\sum_{Q\in\mathcal{W}(G)} \bigg\{  \bigg(    
\int_{\widehat{Q}} \frac{\lvert u_Q(x)\rvert^p}{\mathrm{dist}(x,\partial G)^{\delta p}}\,dx\bigg)^{q/p}
 +\bigg(\int_{Q^*} \int_{Q^*}\frac{\lvert u_Q(x)-u_Q(y)\rvert^{p}}{\lvert x-y\rvert^{n+\delta p}}\,dy\,dx
 \bigg)^{q/p} \bigg\}\,.
\end{align*}
Since $\lvert u_Q\rvert \le \lvert u\rvert$ and $\sum_{Q} \chi_{\widehat{Q}}\le C$, we may
apply 
H\"older's inequality with $(q/p,q/(q-p))$ and the
$(\delta ,q,p)$-Hardy inequality   \eqref{e.hardy} to obtain
\begin{equation}\label{e.etehd}
\begin{split}
\sum_{Q\in\mathcal{W}(G)}\bigg(\int_{\widehat{Q}} \frac{\lvert u_Q(x)\rvert^p}{\mathrm{dist}(x,\partial G)^{\delta p}}\,dx\bigg)^{q/p}
%&\le C \sum_{Q\in\mathcal{W}(G)}\bigg(\int_{\widehat{Q}}\frac{\lvert u(x)\rvert^p}{\mathrm{dist}(x,\partial G)^{\delta p}}\,dx\bigg)^{q/p}\\
&\le C \sum_{Q\in\mathcal{W}(G)}\int_{\widehat{Q}} \frac{\lvert u(x)\rvert^q}{\mathrm{dist}(x,\partial G)^{q(\delta +n(1/q-1/p))}}\,dx\\
&\le C \bigg(\int_G\int_G \frac{\lvert u(x)-u(y)\rvert^p}{\lvert x-y\rvert^{n+\delta p}}\,dy\,dx\bigg)^{q/p}
=C \lvert u\rvert_{W^{\delta,p}(G)}^{q}\,.
\end{split}
\end{equation}
We fix $x,y\in G$ to estimate the remaining series. The following pointwise estimates will be useful,
\begin{align*}
\lvert u_Q(x)-u_Q(y)\rvert %&= \lvert u(x)\varphi_Q(x) - u(x)\varphi_Q(y) + u(x)\varphi_Q(y) - u(y)\varphi_Q(y)\rvert\\
&\le  \lvert u(x)\rvert \lvert \varphi_Q(x)-\varphi_Q(y)\rvert + \lvert u(x)-u(y)\rvert \varphi_Q(y)\\
&\le C\cdot \ell(Q)^{-1}\cdot\lvert u(x)\rvert\cdot \lvert x-y\rvert  + \lvert u(x)-u(y)\rvert\,.
\end{align*}
Namely, since $\sum_{Q\in\mathcal{W}(G)} \chi_{Q^*}\le C\chi_G$ and $q\ge p$, we find that
\begin{align*}
 \sum_{Q\in\mathcal{W}(G)}
\bigg(\int_{Q^*} \int_{Q^*}\frac{\lvert u(x)-u(y)\rvert^{p}}{\lvert x-y\rvert^{n+\delta p}}\,dy\,dx\bigg)^{q/p}
\le C \lvert u\rvert_{W^{\delta,p}(G)}^{q}\,.
\end{align*}
By recalling
that $0<\delta <1$ and 
by estimating as in \eqref{e.etehd}
we obtain
\begin{align*}
&\sum_{Q\in\mathcal{W}(G)}
\bigg(
\ell(Q)^{-p}\int_{Q^*} \lvert u(x)\rvert ^p \int_{Q^*}\frac{\lvert x-y\rvert^p}{\lvert x-y\rvert^{n+\delta p}}\,dy\,dx\bigg)^{q/p}\\
&\le C\sum_{Q\in\mathcal{W}(G)}
\bigg(\ell(Q)^{-\delta p}\int_{Q^*} \lvert u(x)\rvert ^p \,dx\bigg)^{q/p}\\
&\le C\sum_{Q\in\mathcal{W}(G)}
\bigg(\int_{Q^*} \frac{\lvert u(x)\rvert ^p}{\mathrm{dist}(x,\partial G)^{\delta p}}\, dx\bigg)^{q/p}\le C \lvert u\rvert_{W^{\delta,p}(G)}^q\,.
\end{align*}
By collecting the estimates and taking the infimum over all admissible $u$ for $\mathrm{cap}_{\delta,p}(K,G)$
we complete the proof.
\end{proof}

Now we are able to give the proof for Theorem \ref{t.hardy_c}.

\begin{proof}[Proof of Theorem \ref{t.hardy_c}]
The implication from  (A) to (B) is a consequence of Lemma \ref{p.necessary}.
Let us then assume that condition (B) holds.
In order to have 
 inequality 
\eqref{e.hardy}
in $D$, by Theorem~\ref{t.maz'ya} it is enough
to prove that there is a constant $C=C(\delta,p,n,c,N)>0$ such that inequality
\begin{equation}\label{e.maz'ya_app}
\int_K \mathrm{dist}(x,\partial D)^{-q(\delta+n(1/q-1/p))}\,dx \le C\,\mathrm{cap}_{\delta,p}(K,D)^{q/p}
\end{equation}
holds for every compact set $K$ in $D$. 
Let us fix a compact set $K$ in $D$.
We consider a variation of inequality \eqref{e.maz'ya_app}:  
there is a constant $C=C(\delta,p,n,c)>0$ such that
the inequality
 \begin{equation}\label{e.cap_est}
\bigg( \int_{K\cap Q}\dist(x,\partial D)^{-q(\delta+n(1/q-1/p))}\,dx\bigg)^{1/q}
\le C\, \mathrm{cap}_{\delta,p}(K\cap Q,D)^{1/p}
 \end{equation}
holds for every Whitney cube  $Q\in\mathcal{W}(D)$.
To prove inequality \eqref{e.cap_est} we let $u\in C_0(D)$ be a test function such that $u(x)\ge 1$ for every $x\in K\cap Q$.
By the properties of Whitney cubes and Theorem \ref{t.main_emb}  we estimate the left hand side of 
inequality \eqref{e.cap_est} 
\begin{align*}
%\int_{K\cap Q}\dist(x,\partial D)^{-sp}\,dxe
 C \lvert K\cap Q\rvert^{1/q}  \lvert Q\rvert^{-(\delta+n(1/q-1/p))/n} &\le C  \lvert K\cap Q\rvert^{1/q-(\delta+n(1/q-1/p))/n}\\
&\le C\lVert u\rVert_{L^{np/(n-\delta p)}(D)}
\\&\le  C\lvert u\rvert_{W^{\delta,p}(D)}\,.
\end{align*}
Inequality \eqref{e.cap_est} follows when we take the infimum over all admissible functions $u$ for the capacity  $\mathrm{cap}_{\delta,p}(K\cap Q,D)$.

We may now finish the proof by using inequality \eqref{e.cap_est} and 
condition (B) 
\begin{align*}
\int_K \mathrm{dist}(x,\partial D)^{-q(\delta+n(1/q-1/p))}\,dx  &=\sum_{Q\in\mathcal{W}(D)} 
\int_{K\cap Q} \mathrm{dist}(x,\partial D)^{-q(\delta+n(1/q-1/p))}\,dx 
\\&\le C \sum_{Q\in\mathcal{W}(D)} \mathrm{cap}_{\delta,p}(K\cap Q,D)^{q/p} \\& \le CN^{q/p} \mathrm{cap}_{\delta,p}( K,D)^{q/p}\,,
\end{align*}
where $C=C(\delta,p,n,c)$.
The proof is complete.
\end{proof}

\end{document}